\documentclass[12pt]{article}
\usepackage{amsmath,amsfonts,amssymb,amsthm}
\usepackage{enumerate}
\usepackage[pagebackref,colorlinks,citecolor=blue,linkcolor=blue]{hyperref}
\usepackage{amsmath,bbm,amssymb,amsxtra}
\usepackage{caption}
\usepackage{mathrsfs}

\usepackage{geometry}
  \def\Xint#1{\mathchoice
    {\XXint\displaystyle\textstyle{#1}}%
    {\XXint\textstyle\scriptstyle{#1}}%
    {\XXint\scriptstyle\scriptscriptstyle{#1}}%
    {\XXint\scriptscriptstyle\scriptscriptstyle{#1}}%
    \!\int}

    \def\XXint#1#2#3{{\setbox0=\hbox{$#1{#2#3}{\int}$ }
    \vcenter{\hbox{$#2#3$ }}\kern-.6\wd0}}
    \def\dashint{\Xint-}

\newcommand{\bx}{{\partial X}}

\newcommand{\Tr}{{\rm Tr}\,}

\newcommand{\Mod}{{\rm Mod}}
\newcommand{\N}{{\mathbb N}}

\newcommand{\real}{{\mathbb R}}

\newcommand{\rarrow}{\rightarrow}
\newtheorem{thm}{Theorem}[section]
\newtheorem{lem}[thm]{Lemma}
\newtheorem{prop}[thm]{Proposition}
\newtheorem{rem}[thm]{Remark}
\newtheorem{cor}[thm]{Corollary}
\newtheorem{defn}[thm]{Definition}
\newtheorem{example}[thm]{Example}
\numberwithin{equation}{section}

\begin{document}
\title{\Large\bf  Trace and density results on regular trees
\footnotetext{\hspace{-0.35cm}
$2010$ {\it Mathematics Subject classfication}: 46E35, 30L99
\endgraf{{\it Key words and phases}: regular tree, boundary trace, Newtonian space, density }
\endgraf{All authors have been supported by the Academy of Finland via  Centre of Excellence in Analysis and Dynamics Research (project No. 307333). This work was also partially supported by the grant 346300 for IMPAN from the Simons Foundation and the matching 2015-2019 Polish MNiSW fund.}
}
}
\author{Pekka Koskela, Khanh Ngoc Nguyen, and Zhuang Wang}
\date{}
\maketitle
\begin{abstract}
We give characterizations for the existence of traces for first order Sobolev spaces defined on regular trees.
\end{abstract}

\section{Introduction}

 Let $X$ be a rooted $K$-ary tree, $K\ge 1.$ We introduce a metric structure on $X$ by considering each edge of $X$ to be an isometric copy of the unit interval. Then the distance between two vertices is the number of edges needed to connect them and there is a unique geodesic that minimizes this number.  Let us denote the root by 0. If $x$ is a vertex, we define $|x|$ to be the distance between 0 and $x.$ Since each edge is an isometric copy of the unit interval, we may extend this distance naturally to any $x$ belonging to an edge.  We define $\partial X$ as the collection of all infinite geodesics starting at the root 0. Then every $\xi\in\partial X$ corresponds to an infinite geodesic $[0,\xi)$ (in $X$) that is an isometric copy of the interval $[0,\infty).$ Hence $x\to \xi$ along $[0,\xi)$ has a canonical meaning.

Given a function $f$ defined on $X,$ we are interested in the collection of those $\xi\in \partial X$ for which the limit of $f(x)$ exists when $x\to \xi$ along $[0,\xi).$ We begin by equipping $\partial X$ with the natural probability measure $\nu$ so that $\nu(I_x)=K^{-j}$ when $x$ is a vertex with $|x|=j$ and 
$I_x=\{\xi\in\partial X:\ x\in [0,\xi)\}.$ Towards defining the classes of functions that are of interest to us, we define a measure and a new distance function on $X.$ Write $d|x|$ for the length element on $X$ and let $\mu:[0,\infty)\to (0,\infty)$ be a locally integrable function. We abuse notation and refer also to the measure generated via $d\mu(x)=\mu(|x|)d|x|$ by $\mu.$ Further, let $\lambda:[0,\infty)\to (0,\infty)$ be locally integrable and define a distance via $ds(x)=\lambda(|x|)d|x|$ by setting $d(z,y)=\int_{[z,y]}ds(x)$
whenever $z,y\in X$ and $[z,y]$ is the unique geodesic between $z$ and $y.$ For convenience, we assume additionally that $\lambda^p/\mu\in L^{1/(p-1)}_{\rm loc}([0, \infty))$ if $p>1$ below and that $\lambda/\mu\in L^{\infty}_{\rm loc}([0, \infty)$ if $p=1$. Then $(X,d,\mu)$ is
a metric measure space and we let $N^{1,p}(X):=N^{1, p}(X, d, \mu)$, $1\le p<\infty,$ be the associated Sobolev space based on upper gradients \cite{HK98}, as introduced in \cite{N00}. See Section 2 for the precise definition. We show there
that, actually, each $u\in N^{1,p}(X)$  is absolutely continuous
on each edge, with $u'\in L^p_{\mu}(X).$ As usual, $N^{1,p}_0(X)$ is the completion of the family of functions with compact support in $N^{1,p}(X).$

In order to state our results, we need two more concepts. Given $1<p<\infty$ we set
$$R_p=\int_0^{\infty} \lambda(t)^{\frac{p}{p-1}}\mu(t)^{\frac{1}{1-p}}K^{\frac{ t }{1-p}}\, dt=\left\|\frac{\lambda(t)}{\mu(t)^{1/p} K^{t/p}}\right\|^{p-1}_{L^{p/p-1}([0, \infty))}$$
and we define
$$R_1=\left\|\frac{\lambda(t)}{\mu(t)K^{ t }}\right\|_{L^{\infty}([0, \infty) )}.$$
One should view $R_p$ as an isoperimetric profile $(X,d,\mu):$ in case of a Riemannian manifold $M$, 
the natural version of $R_p$ is closely related to the parabolicity of the manifold \cite{troyanov99};
$R_p=\infty$ guarantees  parabolicity (every compact set is of relative $p$-capacity zero). This suggests
that the existence of limits for Sobolev functions along geodesics might be somehow related to finiteness
of $R_p.$ Let us say that the trace of a given function $f$, defined on $X$,  exists if 
\begin{equation}\label{intro-trace}
\Tr f(\xi):=\lim_{[0, \xi)\ni x\rarrow \xi} f(x)
\end{equation}
exists for $\nu$-a.e. $\xi\in\partial X.$ We then denote by $\Tr f$ the trace function of $f$. For other possible definitions of the trace and connections
between them see \cite{khanhzhuang}. 

Our first result gives a rather complete solution for the existence of traces in the case $\mu(X)<\infty.$

\begin{thm}\label{main-thm2}
Let $X$ be a rooted $K$-ary tree with distance $d$ and measure $\mu$. Assume $\mu(X)<\infty$. For $1 \leq p<\infty$, the following are equivalent:

\noindent (i) $R_p<\infty$.

\noindent (ii) $\Tr f$ exists for any $f\in {N}^{1, p}(X)$ and $\Tr:{N}^{1, p}(X) \rightarrow L^p_{\nu}(\bx)$ is a bounded linear operator.

\noindent (iii) $\Tr f$ exists for any $f\in {N}^{1, p}(X)$.

\noindent (iv) $ N^{1, p}_{0}(X)\subsetneq N^{1, p}(X)$.

\end{thm}

In \cite{BBGS,KW19} the trace spaces of our Sobolev spaces were identified as suitable Besov-type spaces for very specific choices of $\mu,\lambda$. 

Our second result deals with the case of infinite volume.

\begin{thm}\label{main-infinity}
Let $X$ be a rooted $K$-ary tree with distance $d$ and measure $\mu$. Assume $\mu(X)=\infty$. For $1 \leq p<\infty$, the following hold:

\noindent (1) $N^{1, p}_{0}(X)=N^{1,p}(X)$.

\noindent (2) For any $f\in N^{1, p}(X)$, if $\Tr f$ exists, then $\Tr f=0$.

\noindent (3) If $p>1$, then $R_p<+\infty$ is  a sufficient condition for $\Tr f$ to exist for any $f\in N^{1, p}(X)$. Furthermore, $R_1<\infty$ if and only if $\Tr f$ exists for any $f\in {N}^{1, 1}(X)$. 
\end{thm}

Theorem \ref{main-infinity} does not claim $R_p<\infty$ to be a necessary condition for the existence of traces. In fact, we give in Section \ref{trace-results} an example of a situation where $R_p=\infty$, $p>1$, but $\Tr f$ still exists for any $f\in N^{1, p}(X)$.

Our third result gives a complete answer in the case of homogeneous norms, see Section 2 for 
the relevant definitions. $\dot N^{1,p}_0(X)$  is the completion of the family of functions with compact support in $\dot N^{1,p}(X).$

\begin{thm}\label{main-thm1}
Let  $X$ be a rooted $K$-ary tree with distance $d$ and measure $\mu$. For $1 \leq p<\infty$, the following are equivalent:

\noindent (i) $R_p<\infty$.

\noindent (ii) $\Tr f$ exists for any $f\in \dot{N}^{1, p}(X)$ and $\Tr:\dot{N}^{1, p}(X) \rightarrow L^p_\nu(\bx)$ is a bounded linear operator.

\noindent (iii) $\Tr f$ exists for any $f\in \dot{N}^{1, p}(X)$.

\noindent (iv) $\dot N^{1, p}_{0}(X)\subsetneq\dot N^{1, p}(X)$.
\end{thm}

Let us close this introduction with some comments on Theorem \ref{main-thm1}.  
Even though the condition $R_p=\infty$ implies $p$-parabolicity, finiteness of 
this quantity does not, in general, prevent $p$-parabolicity, see \cite{G99}. 
Hence Theorem \ref{main-thm1} and the preceding theorems are somewhat 
surprising. In fact, it follows from our results that, in the setting of this paper, $R_p=\infty$ precisely when $(X,d,\mu)$ is $p$-parabolic. See \cite{khanh} for more on this. 
Hence the reader familiar with moduli of curve families might wish to view Theorem \ref{main-thm1} as kind of a version of the equivalence between modulus and 
capacity.

Partial motivation for this paper comes from boundary value problems for the $p$-Laplace equation. For the case of manifolds see \cite{holopainen,H94} and for the setting of metric spaces see \cite{BB11,BBS,KLLS}. Classical trace results on the Euclidean spaces can be found in
\cite{Ar,Ga,HaMa,IV,KTW,MirRus,Pe,Tyu1,Tyu2}  and studies of parabolicity on infinite networks in \cite{s94,Y77}.
For trace results in the metric setting see \cite{BBGS,KW19,LXZ19,Ma,MSS}. Our second motivation comes from the recent paper \cite{ns} where a version of Theorem \ref{main-thm1} was established on regular trees for the case $p=2.$

The paper is organized as follows. In Section \ref{pre}, we introduce  regular trees, boundaries of trees and Newtonian spaces on our trees. We study the trace results in Section \ref{trace-results} and the density results are given in Section \ref{density-results}. In Section \ref{proofs}, we give the  proofs of Theorems \ref{main-thm2}--\ref{main-thm1}.

\section{Preliminaries}\label{pre}
Throughout this paper, the letter $C$ (sometimes with a subscript) will denote positive constants
that usually depend only on the space and may change at
different occurrences; if $C$ depends on $a,b,\ldots$, we write
$C=C(a,b,\ldots)$.
The notation $A\approx B$ means that there is a constant
$C$ such that $1/C\cdot A\leq B\leq C\cdot A$. The notation $A\lesssim B$
($A\gtrsim B$) means that there is a constant $C$ such that
$A\leq  C\cdot B$ ($A\geq C\cdot B$).
\subsection{Regular trees and their boundaries}\label{s-regular}
A {\it graph} $G$ is a pair $(V, E)$, where $V$ is a set of vertices and $E$ is a set of edges.   We call a pair of vertices $x, y\in V$  neighbors if $x$ is connected to $y$ by an edge. The degree of a vertex is the number of its neighbors. The graph structure gives rise to a natural connectivity structure. A {\it tree} is a connected graph without cycles. A graph (or tree) is made into a metric graph by considering each edge as a geodesic of length one.

We call a tree $X$ a {\it rooted tree} if it has a distinguished vertex called the {\it root}, which we will denote by $0$. The neighbors of a vertex $x\in X$ are of two types: the neighbors that are closer to the root are called {\it parents} of $x$ and all other  neighbors  are called {\it children} of $x$. Each vertex has a unique parent, except for the root itself that has none. 

A {\it $K$-ary tree}  is a rooted tree such that each vertex has exactly $K$ children. Then all vertices except the root  of  a $K$-ary tree have degree $K+1$, and the root has degree $K$. In this paper we say that a tree is {\it regular} if it is a $K$-ary tree for some $K\geq 1$.

For $x\in X$, let $|x|$ be the distance from the root $0$ to $x$, that is, the length of the geodesic from $0$ to $x$, where the length of every edge is $1$ and we consider each edge to be an isometric copy of the unit interval. The geodesic connecting two vertices $x, y\in V$ is denoted by $[x, y]$, and its length is denoted $|x-y|$. If $|x|<|y|$ and $x$ lies on the geodesic connecting $0$ to $y$, we write $x<y$ and call the vertex $y$ a descendant of the vertex $x$. More generally, we write $x\leq y$ if the geodesic from $0$ to $y$ passes through $x$, and in this case $|x-y|=|y|-|x|$.

On the $K$-regular tree $X$, for any $n\in \N$, let $X^n$ be a subset of $X$ by setting
\[X^n:=\{x\in X: |x|<n\}.\]

On the $K$-regular tree $X$, we define the metric $ds$ and measure $d\mu$ by setting
\begin{equation*}
d\mu=\mu(|x|)\,d|x|,\ \  ds(x)=\lambda(|x|)\, d|x|,
\end{equation*}
where $\lambda, \mu:[0, \infty)\rightarrow (0, \infty)$ with $\lambda, \mu\in L^1_{\rm loc}([0, \infty))$.   Throughout this paper, we let $1\leq p<\infty$ and  assume additionally that $\lambda^p/\mu\in L^{1/(p-1)}_{\rm loc}([0, \infty))$ if $p>1$ and that $\lambda/\mu\in L^{\infty}_{\rm loc}([0, \infty)$ if $p=1$. 
Here $d\,|x|$ is the measure which gives each edge Lebesgue measure $1$, as we consider each edge to be an isometric copy of the unit interval and the vertices are the end points of this interval. Hence for any two points $z, y\in X$, the distance between them is 
\[d(z, y)=\int_{[z, y]} \,ds(x)=\int_{[z, y]}\lambda(|x|)\, d|x|,\]
where $[z, y]$ is the unique geodesic from $z$ to $y$ in $X$.

We abuse the notation and let $\mu(x)$ and $\lambda(x)$ denote $\mu(|x|)$ and $\lambda(|x|)$,respectively, for any  $x\in X$, if there is no danger of confusion.

Next we construct the boundary of the regular $K$-ary tree.  An element $\xi$ in $\bx$ is identified with an infinite geodesic in $X$ starting at the root $0$. Then we may denote $\xi=0x_1x_2\cdots$, where $x_i$ is a vertex in $X$ with $|x_i|=i$, and $x_{i+1}$ is a child of $x_i$. Given two points $\xi, \zeta\in \bx$, there is an infinite geodesic $[\xi, \zeta]$ connecting $\xi$ and $\zeta$. 

 To avoid confusion, points in $X$ are denoted by Latin letters such as $x, y$ and $z$, while for points in $\bx$ we use Greek letters such as $\xi, \zeta$ and $\omega$. 

We equip $\bx$ with the natural probability measure $\nu$ as in Falconer \cite{F97} by distributing the unit mass uniformly on $\bx$. For any $x\in X$ with $|x|=j$, if { we denote by $I_x$ the set }
$$\{\xi\in \bx: \text{the geodesic $[0, \xi)$ passes through $x$}\},$$
then the measure of $I_x$ is $K^{-j}.$ We refer to \cite[Lemma 5.2]{BBGS} for a more information on our boundary measure $\nu$.
\subsection{Newtonian spaces}\label{newtonian}
Let $X$ be a $K$-regular tree with metric and measure defined as in Section \ref{s-regular}. Let $\mathscr M$ denote the family of all nonconstant rectifiable curves in $X$.  We recall the definition of $p$-modulus of curve families in $\mathscr M$, see \cite{H03,pekka} for more detailed discussions.
\begin{defn}\rm
For $\Gamma\subset \mathscr M$, let $F(\Gamma)$ be the family of all Borel measurable functions $\rho: X\rightarrow [0, \infty]$ such that
\[\int_\gamma g\, ds\geq 1\ \ \text{for every} \ \gamma\in \Gamma.\]
For $1\leq p<\infty$, we define
\[\Mod_p(\Gamma)=\inf_{\rho\in F(\Gamma)} \int_X \rho^p\, d\mu.\]
The number $\Mod_p(\Gamma)$ is called the {\it $p$-modulus} of the family $\Gamma$.
\end{defn}

\begin{prop}\label{modulus-0}
Let $1\leq p<\infty$. On the $K$-regular tree $X$, {the empty family is the only curve family with zero $p$-modulus.} 
\end{prop}
\begin{proof}
First, it follows from the definition that the $p$-modulus of the empty family is zero. 

For a given non-empty curve family $\Gamma$, let $\gamma\in \Gamma$ be 
a rectifiable curve. By picking a subcurve if necessary, we may assume that $\gamma$ is a part of a geodesic ray. Then it follows from \cite[Theorem 5.2 and Lemma 5.3]{H03} that $\Mod_p(\Gamma)\geq \Mod_p(\{\gamma\})$.

For any Borel measurable function $\rho\in F(\{\gamma\})$, we have $\int_\gamma \rho\, ds\geq 1$. By the monotone convergence theorem, we may assume that $\int_\gamma \rho\, ds\geq 1/2$ for a subcurve, still denoted $\gamma,$ that is contained in $\{x\in X: |x|\leq N\}.$
Notice that 
$$ds(x)=\frac{\lambda(x)}{\mu(x)} \, d\mu(x).$$
For $p>1$, it follows from the H\"older inequality that
\begin{align*}
\int_\gamma \rho\, ds=\int_\gamma \rho \frac{\lambda}{\mu} \, d\mu &\leq \left(\int_\gamma \rho^p\, d\mu\right)^{1/p}\left(\int_\gamma \frac{\lambda^{p/(p-1)}}{\mu^{p/(p-1)}}\, d\mu\right)^{(p-1)/p}\\
&\leq C(N, p, \lambda, \mu, N) \left(\int_\gamma \rho^p\, d\mu\right)^{1/p},
\end{align*}
{since it follows from $\lambda^p/\mu\in L^{1/(p-1)}_{\rm loc}([0, \infty))$ that 
\[\int_\gamma \frac{\lambda^{p/p-1}}{\mu^{p/p-1}}\, d\mu \leq \int_{0}^{N}\left(\frac{\lambda(t)^{p}}{\mu(t)}\right)^{\frac{1}{p-1}}\, dt<\infty.\] }
Hence we have that 
\begin{equation}\label{Mod_p}
\int_X \rho^p\, d\mu\geq \int_\gamma \rho^p\, d\mu\geq C(N, p, \lambda, \mu) \left(\int_\gamma \rho\, ds\right)^p\geq C(N, p, \lambda, \mu)/2>0.
\end{equation}
For the case $p=1$, by a similar argument without using the H\"older inequality, it follows from $\lambda/\mu\in L^\infty_{\rm loc}([0, \infty))$ that 
\begin{equation}\label{Mod_1}
\int_X \rho\, d\mu\geq C(N, \lambda, \mu)>0.
\end{equation}
Thus,
\[\Mod_p(\Gamma)\geq \Mod_p(\{\gamma\}) >0\ \ \text {for}\ 1\leq p<\infty,\]
which finished the proof.
\end{proof}

Let $u\in L_{\rm loc}^1(X)$. We say that a Borel function $g: X\rarrow [0, \infty]$ is an {\it upper gradient} of $u$ if  
\begin{equation}\label{gradient}|u(z)-u(y)|\leq \int_{\gamma} g\, ds
\end{equation}
whenever $z, y\in X$ and $\gamma$ is the geodesic from $z$ to $y$. In the setting of a tree any rectifiable curve with end points $z$ and $y$ contains the geodesic connecting $z$ and $y$, and therefore the upper gradient defined above is equivalent to the definition which requires that inequality \eqref{gradient} holds for all rectifiable curves with end points $z$ and $y$. 
In \cite{H03,pekka}, the notion $p$-weak upper gradient is given. A Borel function $g: X\rarrow [0, \infty]$ is called a $p$-weak upper gradient of $u$ if \eqref{gradient} holds on $p$-a.e. curves $\gamma\in \mathscr M$, i.e., \eqref{gradient} holds for all curves $\gamma\in \mathscr M\setminus \Gamma$, where $\Mod_p(\Gamma)=0$.
 Notice that by Proposition \ref{modulus-0}, any $p$-weak upper gradient is actually an upper gradient here. We refer to \cite{H03,pekka} for more information about $p$-weak upper gradients.

The notion of upper gradients is due to Heinonen and Koskela \cite{HK98}; we refer interested readers to \cite{BB11,H03,pekka,N00} for a more detailed discussion on upper gradients.

The {\it Newtonian space} $N^{1, p}(X)$, $1\leq p<\infty$, is defined as the collection of all the functions $u$ for which 
$$\|u\|_{N^{1, p}(X)}:= \|u\|_{L^p(X)}+\inf_g \|g\|_{L^p(X)}<\infty,$$
where the infimum is taken over all upper gradients of $u$. If $u\in N^{1, p}(X)$, then it has a minimal $p$-weak upper gradient, which is an upper gradient in our case (by Proposition \ref{modulus-0}). We denote by $g_u$ the minimal upper gradient, which is unique up to measure zero and which is minimal in the sense that if $g\in L^p(X)$ is any upper gradient of $u$ then $g_u\leq g$ a.e. We refer to \cite[Theorem 7.16]{H03} for proofs of the existence and uniqueness of such minimal upper gradient. Throughout this paper, we denote by $g_u$ the (minimal) upper gradient of $u$.

By Proposition \ref{modulus-0}, it follows from \cite[Definition 7.2 and Lemma 7.6]{H03} that any function $u\in L^1_{\rm loc}(X)$ with an upper gradient $0\leq g\in L^p(X)$  is locally absolutely continuous, for example, absolutely continuous on each edge. Moreover, the ``classical" derivative $u'$ of this locally absolutely continuous function is a minimal upper gradient in the sense that $g_u=|u'(x)|/\lambda(x)$ when $u$ is parametrized in the nature way.

We define the {\it homogeneous Newtonian spaces} $\dot{N}^{1, p}(X)$, $1\leq p<\infty$, the collection of all the continuous functions $u$ that have an
upper gradient $0\leq g\in L^p(X)$, for which the homogeneous $\dot{N}^{1, p}$-norm of $u$ defined as
$$\|u\|_{\dot N^{1, p}(X)}:= |u(0)|+\inf_g  \|g\|_{L^p(X)}$$
is finite. Here $0$ is the root of the regular tree $X$ and the infimum is taken over all upper gradients of $u$.

\section{Trace results}\label{trace-results}
In this section, if we do not specifically mention, we always assume that $X$ is a $K$-regular tree with measure and metric as in Section \ref{s-regular}. 

\begin{lem}\label{lemma3.6}
Let $1 \leq p<\infty$. For any $f\in L^p(X)$, we have that
$$\int_\bx\int_{[0, \xi)}|f(x)|^p K^{j(x)}\, d\mu(x)\, d\nu(\xi) \approx \int_{X} |f(x)|^p\, d\mu(x),$$
where $j(x)$ is the largest integer such that $j(x)\leq |x|$.
\end{lem}
\begin{proof}
Let $f\in L^p(X)$. For any $\xi\in \bx$, let $x_j=x_j(\xi)$ be the ancestor of $\xi$ with $|x_j|=j$. Then it follows from Fubini's Theorem that
\begin{align*}
\int_\bx\int_{[0, \xi)}|f(x)|^p K^{j(x)}\, d\mu(x)\, d\nu(\xi)&=\int_\bx \sum_{j=0}^{+\infty} \int_{[x_j(\xi), x_{j+1}(\xi)]}|f(x)|^p K^{j}\,d\mu(x)\, d\nu(\xi)\\
&=\int_X |f(x)|^p  \int_\bx \sum_{j=0}^{+\infty} K^{j}\chi_{[x_j(\xi), x_{j+1}(\xi)]}(x) \, d\nu(\xi)\, d\mu(x).
\end{align*}
Note that $\chi_{[x_{j}(\xi), x_{j+1}(\xi)]}(x)$ is nonzero only if $j\leq |x|\leq j+1$ and $x<\xi$. Thus the above equality can be rewritten as
$$\int_\bx\int_{[0, \xi)}|f(x)|^p K^{j(x)}\, d\mu(x)\, d\nu(\xi)=\int_X |f(x)|^p K^{j(x)} \nu(I_x)\, d\mu(x),$$
where $I_x=\{\xi\in \bx: x<\xi\}$. Since $\nu(I_x)\approx K^{-j(x)}$, we obtain that
$$\int_\bx\int_{[0, \xi)}|f(x)|^p K^{j(x)}\, d\mu(x)\, d\nu(\xi) \approx \int_{X} |f(x)|^p\, d\mu(x).$$
\end{proof}

\begin{thm}\label{trace}
 Let $1 \leq p<\infty$ and assume that  $R_p<+\infty$.
 Then the trace $\Tr$ in \eqref{intro-trace} gives a bounded linear operator $\Tr: \dot{N}^{1, p}(X)\rightarrow L^p(\bx)$.
 \end{thm}
\begin{proof}
Let $f\in \dot{N}^{1,p}(X)$.  Our task is to show that 
\begin{equation}\label{trace-operator}
\Tr f(\xi):=\tilde f(\xi)=\lim_{[0, \xi)\ni x\rarrow \xi} f(x),
\end{equation}
exists for $\nu$-a.e. $\xi\in \bx$ and that the trace $\Tr f$ satisfies the norm estimates.

To show that the limit in \eqref{trace-operator} exists for $\nu$-a.e. $\xi\in \bx$, it suffices to show that the function 
\begin{equation}\notag
{\tilde f}^*(\xi)=|f(0)|+\int_{[0, \xi)} g_f\, ds
\end{equation}
is in $L^p(\bx)$, where $[0, \xi)$ is the geodesic ray from $0$ to $\xi$ and $g_f$ is an upper gradient of $f$. To be more precise, if $\tilde f^*\in L^p(\bx)$, we have $|\tilde f^*|<\infty$ for $\nu$-a.e. $\xi\in \bx$, and hence the limit in $\eqref{trace-operator}$ exists for $\nu$-a.e. $\xi\in\bx$.

Since we have
\[ds =\frac{\lambda(x)}{\mu(x)}\, d\mu,\]
 we obtain the estimate
\begin{align}
{\tilde f}^*(\xi) &=|f(0)|+\int_{[0, \xi)} g_f\, ds =|f(0)|+\int_{[0, \xi)} g_f \frac{\lambda(x)}{\mu(x)}\, d\mu. \label{estimate-f}
\end{align}

For $p>1$, it follows from the H\"older inequality  that
\begin{align*}
|{\tilde f}^*(\xi)|^p &\lesssim |f(0)|^p+\left(\int_{[0, \xi)} {g_f} K^{j(x)/p} \frac{\lambda(x)}{\mu(x)K^{j(x)/p}}  \, d\mu\right)^p\\
& \leq |f(0)|^p+ \int_{[0, \xi)} {g_f}^p K^{j(x)}\, d\mu \left(\int_{[0, \xi)} \left(\frac{\lambda(x)}{\mu(x)K^{j(x)/p}}\right)^{\frac{p}{p-1}}\, d\mu\right)^{p-1}\\
&\leq |f(0)|^p + {R_p}^{p-1} \int_{[0, \xi)} {g_f}^p K^{j(x)}\, d\mu,
\end{align*}
where $j(x)$ is the largest integer such that $j(x)\leq |x|$. Here the last inequality holds since
\[\int_{[0, \xi)} \left(\frac{\lambda(x)}{\mu(x)K^{j(x)/p}}\right)^{\frac{p}{p-1}}\, d\mu\approx \int_{0}^\infty \frac{\lambda(t)^{\frac{p}{p-1}}}{\mu(t)^{\frac{p}{p-1}}K^{\frac{t}{p-1}}} \mu(t)\, dt=R_p.\]

Integrating over all $\xi\in \bx$, since $\nu(\bx)= 1$, $R_p<+\infty$ and $g_f\in L^p(X)$, it follows from Lemma \ref{lemma3.6} that
\begin{align}
\int_\bx |{\tilde f}^*(\xi)|^p\, d\nu &\lesssim |f(0)|^p+\int_\bx \int_{[0, \xi)} {g_f(x)}^p K^{j(x)}\, d\mu(x)\, d\nu(\xi)\notag\\
&\lesssim |f(0)|^p+ \int_X g_f(x)^p\, d\mu(x), \ \ p>1.\label{eq1}
\end{align}

For $p=1$, integrating over all $\xi\in\bx$ with respect to estimate \eqref{estimate-f}, since $\nu(\bx)=1$, we obtain by means of Fubini's theorem that
\begin{align*}
\int_{\bx} |\tilde f^*(\xi)|\, d\nu &\leq |f(0)| + \int_{\bx} \int_{X} g_f(x) \chi_{[0,\xi)}(x)\frac{\lambda(x)}{\mu(x)} \, d\mu(x)\, d\nu(\xi)\\
& =|f(0)| +\int_X g_f(x) \frac{\lambda(x)}{\mu(x)} \int_\bx \chi_{[0,\xi)}(x) \, d\nu(\xi)\, d\mu(x)\\
& = |f(0)| +\int_X g_f(x) \lambda(x)\mu(x)^{-1} \nu(I_x)\, d\mu(x).
\end{align*}
Here in the above estimates, the notations $I_x$ and $j(x)$ are the same ones as those we used in Lemma \ref{lemma3.6}. Since $\nu(I_x)\approx K^{-j(x)}\approx K^{-|x|}$ and $R_1<+\infty$, we further obtain that
\begin{equation}\label{eq2}
\int_{\bx} |\tilde f^*(\xi)|\, d\nu\leq |f(0)| + R_1 \int_X g_f(x) \, d\mu(x)\lesssim |f(0)| + \int_X g_f(x) \, d\mu(x).
\end{equation}

Hence we obtain from estimates \eqref{eq1} and \eqref{eq2} that ${\tilde f}^*$ is in $L^p(\bx)$ for $1 \leq p<\infty$, which gives the existence of the limits in \eqref{trace-operator} for $\nu$-a.e. $\xi\in\bx$. In particular, since $|\tilde f|\leq {\tilde f}^*$, we have the estimate
$$\int_\bx |\tilde f|^p\, d\nu\lesssim |f(0)|^p + \int_X {g_f}^p\, d\mu,$$
and hence the norm estimate
\begin{equation}\notag
\|\tilde f\|_{L^p(\bx)}\lesssim |f(0)|+\left( \int_X {g_f}^p\, d\mu\right)^{1/p}=\|f\|_{\dot{N}^{1, p}(X)}.
\end{equation}
\end{proof}

Since every $f\in N^{1, p}(X)$ is locally absolutely continuous, a direct computation gives the estimate $|f(0)|\lesssim \|f\|_{N^{1, p}(X)}$. Hence we obtain 
the following result from the above theorem.

\begin{cor}\label{trace-pp}
 Let $1 \leq p<\infty$ and assume that  $R_p<+\infty$.
 Then the trace $\Tr$ in  \eqref{intro-trace} gives a bounded linear operator $\Tr: {N}^{1, p}(X)\rightarrow L^p(\bx)$ .
\end{cor}

Next, we study non-existence of the traces when $R_p=\infty$. Before going to the main theorems, we introduce the following lemma. 
\begin{lem}[\cite{V85}]\label{lemma3.1}
Let $(\Omega, d, \mu_\Omega)$ be a $\sigma$-finite metric measure space. Then the following conditions on $(\Omega, d, \mu_\Omega)$ are equivalent:

\item (i) $L^p(\Omega)\subset L^q(\Omega)$ for all $p, q\in (0, \infty)$ with $p>q$;

\item (ii) $\mu(\Omega)<+\infty$.
\end{lem}

\begin{thm}\label{trace-example}
  Let $1 \leq p<\infty$ and assume that  $R_p=+\infty$. Then there exists a function $u\in \dot{N}^{1, p}(X)$ such that 
 \begin{equation}\label{eq3.1}\lim_{[0, \xi)\ni x\rarrow \xi} u(x)=+\infty,\ \ \forall \  \xi\in \bx.\end{equation}
 \end{thm}
\begin{proof}
To construct the function $u\in \dot{N}^{1, p}(X)$ satisfying \eqref{eq3.1}, it suffices to find a nonnegative measurable function $g: [0, \infty)\rightarrow [0, \infty]$   such that
\begin{equation}\label{eq3.31}\left\{
\begin{array}l
\int_{0}^{+\infty} g(t) \lambda(t)\, dt=+\infty\\
\int_{0}^{+\infty} g(t)^p\mu(t) K^t\, dt<+\infty.
\end{array}\right.
\end{equation}
Given such $g$, we may define the function $u$ by setting $u(0)=0$ and $u(x)=\int_{0}^{|x|} g(t) \lambda(t)\, dt$ for any $x\in X$. Then it follows from the definition of upper gradient that $g_u: X\rightarrow [0, \infty]$ defined by $g_u(x)=g(|x|)$ is an upper gradient of $u$. Moreover, we obtain that
\[\|g_u\|^p_{L^p(X)}=\int_{X} {g_u}^p\, d\mu\approx \int_{0}^{+\infty} {g(t)}^p \mu(t) K^{t}\, dt<+\infty.\]
Hence the condition \eqref{eq3.31} implies $u\in \dot{N}^{1, p}(X)$  and that \eqref{eq3.1} holds.  

For $p=1$, since $R_1=\left\|\frac{\lambda(t)}{\mu(t)K^{ t }}\right\|_{L^{\infty}([0, \infty) )}=\infty$, the  sets
$$A_k:=\left\{t\in [0, \infty): \frac{\lambda(t)}{\mu(t)K^t}\geq 2^k\right\},\ \ \ k\in \N $$
form a nonincreasing sequence of subsets of $[0, \infty)$ and we have 
\[|A_k|>0\ \ \text{for any}\ \ k, n\in\mathbb N.\]
Hence there exists an infinite sequence $\{k_n\}_{k_n\in\mathbb N}$ such that $|B_{k_n}|>0$ for any $k_n$, where
\[B_{k_n}=A_{k_n}\setminus A_{k_{n+1}}=\left\{t\in [0, \infty): 2^{k_{n+1}}\geq \frac{\lambda(t)}{\mu(t)K^t}> 2^{k_n}\right\};\]
 otherwise, there will be  $N\in\N$ such that for any $k\geq N$, we have $|B_{k}|=0$, and hence $|A_k|=0$ for any $k\geq N$, which is a contradiction. Since $\lambda\in L^1_{\rm loc}$, we may also assume that $0<\int_{B_{n_k}} \lambda(t)\, dt<+\infty$ by replacing $B_{n_k}$ with a suitable bounded subset if necessary.
Then we define $g$ by setting
\begin{equation*}g(t)=\left\{
\begin{array}{cc}
\frac{1}{\int_{B_{n_k}} \lambda(t)\, dt}&\ \ \forall \ \ t\in B_{n_k},  \ n_k\in \N;\\
0,& \ \ \ \text{otherwise}.
\end{array}\right.
\end{equation*}
It follows  that
\[\int_{0}^{+\infty} g(t) \lambda(t)\, dt=\sum_{n_k} \int_{B_{n_k}} \frac{\lambda(t)}{\int_{B_{n_k}} \lambda(t)\, dt}\, dt=\sum_{n_k} 1=+\infty\]
and from the definition of $B_{n_k}$ that
\begin{align*}
\int_{0}^{+\infty} g(t)\mu(t) K^t\, dt& =\sum_{n_k} \int_{B_{n_k}} \frac{\mu(t)K^t}{\int_{B_{n_k}} \lambda(t)\, dt}\, dt\\&\leq \sum_{n_k} 2^{-n_k} \int_{B_{n_k}} \frac{\lambda(t)}{\int_{B_{n_k}} \lambda(t)\, dt}\, dt
 =\sum_{n_k} 2^{-n_k}<\infty.
 \end{align*}
 Hence \eqref{eq3.31} is satisfied.

For $p>1$, let $r(t)=\lambda(t)^{\frac{p}{p-1}}\mu(t)^{\frac{1}{1-p}} K^{\frac{t}{1-p}}$. Then we know that
\begin{equation}\label{eq3.32}
R_p= \int_{0}^{\infty} \lambda(t)^{\frac{p}{p-1}}\mu(t)^{\frac{1}{1-p}} K^{\frac{t}{1-p}}=\int_{0}^\infty r(t)\, dt =\infty.
\end{equation}
Since $\lambda^p/\mu\in L^{1/(p-1)}_{\rm loc}([0,\infty))$, we have $r\in L^1_{\rm loc}([0, \infty))$.
 Define the function $g$ by setting
\[g(t)=\lambda(t)^{\frac{1}{p-1}}\mu(t)^{\frac{1}{1-p}} K^{\frac{t}{1-p}} \alpha(t)=r(t)\alpha(t)/\lambda(t),\]
where $\alpha: [0, \infty)\rightarrow [0, \infty]$ is be determined. Then to find a function $g$ satisfying \eqref{eq3.31}, it suffices to show the existence of a function $\alpha$ satisfying
\begin{equation}\label{eq3.33}
\left\{
\begin{array}l
\int_{0}^{+\infty} g(t) \lambda(t)\, dt=\int_0^\infty \alpha(t)r(t)\, dt=+\infty\\
\int_{0}^{+\infty} g(t)^p\mu(t) K^t\, dt=\int_0^\infty {\alpha(t)}^p r(t)\, dt<+\infty.
\end{array}\right.
\end{equation}
Consider the metric measure space $([0, \infty), d_{E}, \mu_r)$ with $d_E$ the Euclidean distance where $d\mu_r=r(t)\, dt.$ Since $r\in L^1_{\rm loc}$, we have that
$([0, \infty), d_{E}, \mu_r)$ is a $\sigma$-finite metric measure space. Then it follows from \eqref{eq3.32} that $\mu_r([0, \infty))=+\infty$. Hence by Lemma \ref{lemma3.1}, we know that $L^p([0, \infty), \mu_r)\nsubseteq L^1([0, \infty), \mu_r)$, i.e., there exists a function $\alpha: [0, \infty)\rightarrow \real$ such that $\alpha\in L^p([0, \infty), \mu_r)$ but $\alpha\notin L^1([0, \infty), \mu_r).$ Choosing this $\alpha$ ensures \eqref{eq3.33}.

In conclusion, for $1\leq p<\infty$, we can construct a function $u\in \dot{N}^{1, p}(X)$ satisfying \eqref{eq3.1}. 
\end{proof}

\begin{rem}\rm\label{remark_osc}
If additionally $\mu(X)<\infty$, instead of constructing the above increasing function, we may easily modify the construction so as to obtain a  piecewise monotone function $u\in N^{1, p}(X)$ with values in $[0, 1]$ so that $u(x)=1$ when $|x|=t_{2j}$ and $u(x)=0$ when $|x|=t_{2j+1}$, where $t_k\rightarrow \infty$ as $k\rightarrow \infty$.
Then this oscillatory function $u$ belongs to $N^{1, p}(X)$, but has no limit along any geodesic ray. Hence we obtain the following result.
\end{rem}

\begin{prop}\label{trace-N-example}
 Let $1 \leq p<\infty$ and assume that  $R_p=+\infty$. If $\mu(X)<\infty$, then there exists a function $u\in N^{1, p}(X)$ such that
 $\lim_{[0, \xi)\ni x\rarrow \xi} u(x)$ {does not exist}, {for any} $\xi\in \bx.$

\end{prop}

The above results give the full answers to the trace results for the  homogeneous Newtonian space $\dot N^{1, p}(X)$ and also for the Newtonian space $N^{1, p}(X)$ when $\mu(X)<\infty$. We continue towards the case $\mu(X)=\infty$.

\begin{prop}\label{trace-0}
Let $1 \leq p<\infty$ and assume $\mu(X)=\infty$. Then for any $f\in L^{p}(X)$, we have
\begin{equation}\label{eq-prop3.1}
\liminf_{[0, \xi)\ni x\rightarrow \xi} |f(x)|=0, \ \ \text{ for a.e.} \ \xi\in \bx,
\end{equation}
and hence $\Tr f =0$ if $\Tr f$ exists.
\end{prop}

\begin{proof}
Assume that \eqref{eq-prop3.1} is false. Then there exist a function $f\in  L^{p}(X)$ and a set $E\subset \bx$ with $\nu(E)>0$ such that 
$$\liminf_{[0, \xi)\ni x\rightarrow \xi} |f(x)|>0, \ \ \forall\ \xi\in E.$$
Hence for any $\xi\in E$, there exist a constant $\epsilon(\xi)>0$ and an integer $N(\epsilon)$ such that 
$$|f(x)|\geq \epsilon(\xi)>0,\ \ \ \forall \ x\in [0, \xi) \ \ \text{with} \ \ |x|\geq N(\epsilon).$$
It follows from Lemma \ref{lemma3.6} that
\begin{align*}
\|f\|_{L^p(X)}^p&=\int_X |f(x)|^p\, d\mu \approx \int_\bx\int_{[0, \xi)}|f(x)|^p K^{j(x)}\, d\mu(x)\, d\nu(\xi)\\
&\geq \int_E\int_{[0, \xi)\cap\{|x|\geq N(\xi)\}}|f(x)|^p K^{j(x)}\, d\mu(x)\, d\nu(\xi)\\
&\geq \int_E\int_{[0, \xi)\cap\{|x|\geq N(\xi)\}} \epsilon(\xi)^p K^{j(x)}\, d\mu(x)\, d\nu(\xi)\\
&= \int_E \epsilon(\xi)^p \int_{N(\xi)}^{\infty} K^{j(t)} \mu(t)\, dt\, d\nu(\xi),
\end{align*}
where $j(t)$ is the largest integer such that $j(t)\leq t$.
Since $\mu(X)=\infty$ and $\mu\in L^1_{\rm loc}(X)$,  for any integer $N(\xi)$, we have 
$$\int_{N(\xi)}^{\infty} K^{j(t)} \mu(t)\, dt=\infty.$$
Since $\epsilon(\xi)>0$ for any $\xi\in E$ and $\nu(E)>0$, we obtain that
$$\|f\|_{L^p(X)}^p=+\infty,$$
which contradicts the fact that $f\in  L^{p}(X)$. Thus \eqref{eq-prop3.1} holds.

If $\Tr f$ exists, then $\Tr |f|$ also exists. It follows from the definition of the trace \eqref{intro-trace} and \eqref{eq-prop3.1} that $\Tr |f|=0$.  Hence $\Tr f=0$.
\end{proof}

\begin{prop}\label{no-trace-infinity}
 Assume  $R_1=+\infty$. Then  there exists a function $u\in N^{1,1}(X)$ such that
 $\lim_{[0, \xi)\ni x\rarrow \xi} u(x)$ {does not exist}, {for any} $\xi\in \bx.$
\end{prop}
\begin{proof}
Since $\lambda/\mu\in L^\infty_{\rm loc}([0, \infty))$ implies that 
 $\left\|\frac{\lambda(t)}{\mu(t)K^{ t }}\right\|_{L^{\infty}([0, n) )}<\infty$ for any $n\in \N$, it follows from 
$R_1=\left \| \frac{\lambda(t)}{\mu(t)K^{t}}\right \|_{{L^1({[0,\infty)})}}=\infty$ that  the sequence of sets 
	\[E^n_k:=\left \{t\in [n,\infty): \frac{\lambda(t)}{\mu(t)K^{t}}\geq 2^k \right \}, \ \ n, k\in\mathbb{N} 
	\] 
satisfies
\[|E^n_k|>0\ \ \text{for any}\ \ n,  k\in\mathbb N.\]
Hence we may choose a sequence $\{t_k: t_k\in [0, \infty)\}_{k\in\N}$ with 
\begin{equation}\label{E_k-positive}
t_k\rightarrow \infty \ \ \text{as}\ \ k\rightarrow \infty\ \ \text{and}\ \ |E_k\cap [t_{k-1}, t_k]|>0 \ \ \text{for any}\ \  k\in\mathbb N_+,
\end{equation}
{ where} {$E_k=E_k^0$}. Since $\mu\in L^1_{\rm loc}([0, \infty))$, we have that for any $k\in \N_+$,
\[0<\int_{t_{k-1}}^{t_k} \mu(t)K^t\, dt=:M_k<\infty.\]
By the absolute continuity of integral with respect to measure, we may divide the interval $[t_{k-1}, t_k]$ into $\lceil 2^k M_k \rceil$  subintervals $\{I_j\}_j$ whose interiors are pairwise disjoint such that
\begin{equation}\label{I_k}
\bigcup_{j=1}^{\lceil 2^k M_k \rceil} I_j=[t_{k-1}, t_k]\ \ \text{and}\ \ 0<\int_{I_j} \mu(t)K^t\, dt\leq 2^{-k}.
\end{equation}
Since $|E_k\cap[t_{k-1}, t_k]|>0$ from \eqref{E_k-positive}, we obtain there is at least one subinterval $I_k\in \{I_j\}_j$ such that $|E_k\cap I_k|>0$. Then we define a function $g$ by setting
\begin{equation*}g(t)=\left\{
\begin{array}{cc}
\frac{2}{\int_{E_k\cap I_k}\lambda(t)\, dt}&\ \ \forall \ \ t\in E_k\cap I_k,  \ k\in \N;\\
0,& \ \ \ \text{otherwise}.
\end{array}\right.
\end{equation*}
Since $\lambda(t)$ is always positive and $\lambda\in L^1_{\rm loc}$, the above definition is well-defined. Next we construct the function $u$. For any $k\in \N_+$, since we have
\begin{equation}\label{2^k-estimate}
\int_{t_{k-1}}^{t_k} g(t)\lambda(t)\, dt=\int_{E_k\cap I_k} \frac{2\lambda(t)}{\int_{E_k\cap I_k}\lambda(t)\, dt}\, dt=2,
\end{equation}
we may apply the same idea of constructing as in Remark \ref{remark_osc} on $\{x\in X: t_{k-1}\leq|x|\leq t_k\}$ to obtain  a piecewise monotone function $u$ with upper gradient $g_u(x)=g(|x|)$ and with values in $[0, 1]$ so that $u(x)=0$ when $|x|=t_{k-1}, t_k$ and $u(x)=1$ when $|x|=t'_k$ where $t_{k-1}<t'_k<t_k$.  Then the function $u$ has no limit along any geodesic rays. 

Thus it remains to show that $u\in N^{1,1}(X)$. We first estimate the $L^1$-norm of the upper gradient $g_u$ of $u$. By the definitions of function $g$ and of $E_k$, it follows from estimate \eqref{2^k-estimate} that 
\begin{align*}
\int_{X} g_u\, d\mu&=\int_{0}^{\infty} g(t)\mu(t)K^{t} \, dt =\sum_{k\in\N_+} \int_{E_k\cap I_k}  g(t)\mu(t)K^{t}\, dt\\
&\leq  \sum_{k\in\N_+} 2^{-k} \int_{E_k\cap I_k} g(t)\lambda(t) \, dt = \sum_{k\in\N_+} 2^{1-k}<\infty.
\end{align*}
For the $L^1$-norm estimate of $u$, notice that $u(x)>0$ only if $|x|\in I_k$ for some $k\in \N$. Since $|u(x)|\leq 1$,  we obtain from \eqref{I_k} that
\[\int_X u\, d\mu=\sum_{k\in\N_+}\int_{\{x\in X: |x|\in I_k\}} u(x)\, d\mu(x)\lesssim \sum_{k\in \N_+} \int_{I_k} \mu(t)K^{t} \, dt \leq \sum_{k\in\N_+} 2^{-k}<\infty.\]
We conclude that $u\in N^{1,1}(X)$ and that $\lim_{[0, \xi)\ni x\rarrow \xi} u(x)$ {does not exist}, {for any} $\xi\in \bx.$
\end{proof}

\begin{rem}\rm
The above proposition is strong and surprising, since it does not require 
$\mu(X)<\infty$ anymore. Especially, for $p=1$, it follows from the above proposition, Theorem \ref{trace} and Corollary \ref{trace-pp} that $R_1<\infty$ is a characterization for the existence of traces for both $N^{1,1}(X)$ and $\dot N^{1, 1}$, no matter whether the total measure is finite or not. In the following, we will show that if the total measure is infinite, then $R_p<\infty$ is not a characterization for the existence of traces for $N^{1, p}(X)$ when $1<p<\infty$.
\end{rem}

For simplicity, we consider the special case where $\lambda$ and $\mu$ are 
piecewise constant. More precisely,  assume that 
\[\lambda(t)=\lambda_j, \ \mu(t)=\mu_j, \ \text{for} \ t\in [j, j+1), \ j\in\N,\]
where $\{\lambda_{j}\}_{j\in\N}$ and $\{\mu_{j}\}_{j\in\N}$ are two sequences of positive and finite real numbers.
Then 
\begin{equation}\label{piece-wise-constant}
ds=d\,\lambda(z)=\lambda_j d\, |z| \ \ {\rm and }\ \ d\, \mu(z) = \mu_j d\, |z|, \ \ {\rm for } \ \ j\leq |z| < j+1, j\in \N.
\end{equation}
We establish the following trace result.

\begin{lem}\label{thm-infinity}
Let $X$ be a $K$-regular tree with measure and metric as in \eqref{piece-wise-constant}. Then the following hold:

\item (i)\ \ $\Tr f$ exists and $\Tr f=0$ for any $f\in N^{1, p}(X)$ if $$\sup_j\left\{\max\left\{\frac1{K^j\mu_j} , \frac{{\lambda_j}^p} {K^j\mu_j}\right\}\right\}<+\infty .$$  

\item (ii)\ \  There exists a function $u\in N^{1, p}(X)$ such that $\Tr u$ does not exists if 
$$\sup_j\left\{\min\left\{\frac1{K^j\mu_j}, \frac{{\lambda_j}^p} {K^j\mu_j}\right\}\right\}=+\infty .$$
\end{lem}
\begin{proof}
 (i) Let $f\in N^{1, p}(X)$. Then $f$ and the upper gradient $g_f$ of $f$ belong to $L^p(X)$. It follows from Lemma \ref{lemma3.6} that
$$\int_\bx\int_{[0, \xi)}|f(x)|^p K^{j(x)}\, d\mu(x)\, d\nu(\xi) \approx \int_{X} |f(x)|^p\, d\mu(x)<+\infty$$
and that
$$\int_\bx\int_{[0, \xi)}|g_f(x)|^p K^{j(x)}\, d\mu(x)\, d\nu(\xi) \approx \int_{X} |g_f(x)|^p\, d\mu(x)<+\infty.$$
Hence for a.e. $\xi\in \bx$, we have
\begin{equation}\label{thm1.8-1}
\int_{[0, \xi)}|f(x)|^p K^{j(x)}\, d\mu(x)<+\infty \ \ \text{and}\ \ \int_{[0, \xi)}|g_f(x)|^p K^{j(x)}\, d\mu(x)<+\infty.
\end{equation}
Let $x_j=x_j(\xi)$ be the ancestor of $\xi$ with $|x_j|=j$. By using the H\"older inequality, we obtain that
$$\left(\dashint_{[x_j, x_{j+1}]} |f(x)|\, d\mu(x)\right)^p\leq \frac{1}{\mu_j K^j} \int_{[x_j, x_{j+1}] }|f(x)|^pK^{j(x)}\, d\mu$$
and that
\begin{align*}
\left(\int_{[x_j, x_{j+1}]} |g_u|\, ds\right)^p=\left(\int_{[x_j, x_{j+1}]} |g_u| \lambda_j/\mu_j\, d\mu\right)^p \leq \frac{\lambda_j^p}{\mu_jK^j}\int_{[x_j,x_{j+1}]}|g_u|^pK^{j(x)}d\mu 
\end{align*}
If $\sup_j\left\{\max\left\{\frac1{K^j\mu_j} , \frac{{\lambda_j}^p} {K^j\mu_j}\right\}\right\}<+\infty$, it follows from \eqref{thm1.8-1} that for a.e. $\xi\in \bx$,
$$\dashint_{[x_j(\xi), x_{j+1}(\xi)]} |f(x)|\, d\mu(x)\rightarrow 0\ \ \text{and}\ \ \int_{[x_j(\xi), x_{j+1}(\xi)]} |g_u|\, ds\rightarrow 0,$$
as $j\rightarrow \infty$. Since for any $y, z\in [x_j(\xi), x_{j+1}(\xi)]$, we have
$$|f(y)-f(z)|\leq \int_{[x_j(\xi), x_{j+1}(\xi)]}g_u\, ds\leq \int_{[x_j(\xi), x_{j+1}(\xi)]} |g_u|\, ds,$$
we obtain that
$$\sup_{z\in [x_j(\xi), x_{j+1}(\xi)]} |f(z)| \leq \dashint_{[x_j(\xi), x_{j+1}(\xi)]} |f(x)|\, d\mu(x)+\int_{[x_j(\xi), x_{j+1}(\xi)]} |g_u|\, ds \rightarrow 0,$$
as $j\rightarrow \infty$ for a.e. $\xi\in\bx$. Hence, for a.e. $\xi\in\bx$, 
$$\limsup_{[0, \xi)\ni x\rightarrow \xi} |f(x)| =0,$$
which implies $\Tr |f| =0$. Thus, $\Tr f$ exists and $\Tr f=0$.

(ii) If $\sup_j\left\{\min\left\{\frac1{K^j\mu_j}, \frac{{\lambda_j}^p} {K^j\mu_j}\right\}\right\}=+\infty$, then there exists an increasing sequence $\{n_k\}_{k\in\N}$ such that
$$\frac1{K^{n_k}\mu_{n_k}}\geq 2^k,\ \ \ \frac{{\lambda_{n_k}}^p} {K^{n_k}\mu_{n_k}}\geq 2^k\ \ \forall \ k\in \N.$$
For any $k\in \N$, any edge $[x_1, x_2]\subset X$ with $|x_1|=n_k$ and $|y_1|=n_k+1$, let $z\in[x_1, x_2]$ be the midpoint. Then we define  $u(x_1)=u(x_2)=0$, $u(z)=1$ and extend $u$ linearly on $[x_1, z]$ and $[z, x_2]$.  For any $y\in X$ with $ |y|\notin [n_k, n_{k+1}]$ for any $k\in \N$, define $u(y)=0$. Then the $L^p$-estimate of $u$ is 
$$\int_X |u|^p\, d\mu=\sum_{k\in\N} \int_{\{x\in X: n_k\leq |x|\leq n_k+1\}} |u|^p\, d\mu\leq \sum_{k\in\N} \mu_{n_k}K^{n_k}\leq \sum_{k\in\N} 2^{-k}<+\infty.$$
Moreover, since $u$ is linear on $[x_1, z]$ and on $[z, x_2]$, we have $g_u\leq 2/\lambda_{n_k}$. Hence
$$\int_X |g_u|^p\, d\mu =\sum_{k\in\N} \int_{\{x\in X: n_k\leq |x|\leq n_k+1\}} |g_u|^p\, d\mu \lesssim 2^p\sum_{k\in\N} \frac{K^{n_k}\mu_{n_k}}{{\lambda_{n_k}}^p} \leq \sum_{k\in\N} 2^{p-k}<+\infty.$$
Thus, $u\in N^{1, p}(X)$ and it is easy to check that $\lim_{[0, \xi)\ni x\rarrow \xi} u(x)$ {does not exist}, {for any} $\xi\in \bx.$ 
\end{proof}

Let us give a concrete example.
\begin{example}\rm\label{not-necessary}
Let $\mu_j=K^{-j}$ and $\lambda_j=j^{\frac{1-p}{p}}$. Then $\mu(X)=+\infty$ and
$$\sup_j\left\{\max\left\{\frac1{K^j\mu_j}, \frac{{\lambda_j}^p} {K^j\mu_j}\right\}\right\} =1 <+\infty.$$
Hence it follows from Lemma \ref{thm-infinity} that $\Tr f$ exists for each $f\in N^{1, p}(X)$. But 
for $p>1$,  we have
$$R_p=\sum_{j=0}^{+\infty} \frac{1}{j}=+\infty.$$ 
Thus, for $p>1$, $R_p<+\infty$ is not a necessary condition for $\Tr f$ to exist for all $f\in N^{1, p}(X)$ when $\mu(X)=+\infty$.
\end{example}

The following example implies that the conditions obtained in Lemma \ref{thm-infinity} are not  characterizations of the existence or non-existence of traces.
\begin{example}\label{example1}\rm
Let $1<p<\infty$ and fix a  $K$-regular tree $X$. We construct measures $\mu^1=\{\mu^1_j\}_{j\in\N}$, $\mu^2=\{\mu^2_j\}_{j\in\N}$ and $\lambda^1=\{\lambda^1_j\}_{j\in\N}$, $\lambda^2=\{\lambda^2_j\}_{j\in\N}$ respectively, satisfying $\mu^1(X_1)=\infty=\mu^2(X_2)$ and  
$$\sup_j\left\{\max\left\{\frac1{K^j\mu^i_j} , \frac{{(\lambda^i_j)}^p} {K^j\mu^i_j}\right\}\right\}=+\infty, \ \  \sup_j\left\{\min\left\{\frac1{K^j\mu^i_j}, \frac{{(\lambda^i_j)}^p} {K^j\mu^i_j}\right\}\right\}<+\infty, \ \ \forall\  i=1, 2,$$
such that $\Tr f$ exists for any $f\in N^{1, p}(X, ds^1, d\mu^1)$ and that there exists a function $u\in N^{1, p}(X, ds^2, d\mu^2)$ for which $\Tr u$ does not 
exist.

Let $\mu_j^1=K^{-j}j^{-1}$. Then 
$$\mu_1(X)=\sum_{j\in\N} \mu_j^1 K^j=\sum_{j\in\N}j^{-1}=+\infty$$ and 
$$\sup_j\left\{\max\left\{\frac1{K^j\mu^1_j} , \frac{{(\lambda^1_j)}^p} {K^j\mu^1_j}\right\}\right\}\geq \sup_j\left\{\frac1{K^j\mu^1_j}\right\}=+\infty, \ \ \ \forall \ i=1, 2.$$
Let $(\lambda_j^1)^p=2^{-j} j^{-1}$. Then we obtain
$$\sup_j\left\{\min\left\{\frac1{K^j\mu^1_j} , \frac{{(\lambda^1_j)}^p} {K^j\mu^1_j}\right\}\right\}\leq \sup_j\left\{\frac{{(\lambda^1_j)}^p} {K^j\mu^1_j}\right\}<+\infty.$$ 
Hence $(X, ds^1, d\mu^1)$ satisfies the asserted conditions 
and $R_p((X, ds^1, d\mu^1))<+\infty$. It follows from Corollary \ref{trace-pp} that $\Tr f$ exists for any $f\in N^{1, p}(X, ds^1, d\mu^1)$.

For $X_2$, let $\mu_j^2=K^{-j}$ and $\lambda_j^2=2^j$. Then we have
$$\sup_j\left\{\max\left\{\frac1{K^j\mu^2_j} , \frac{{(\lambda^2_j)}^p} {K^j\mu^2_j}\right\}\right\}\geq \sup_j\left\{\frac{{(\lambda^2_j)}^p} {K^j\mu^2_j}\right\}=+\infty$$
and
$$\sup_j\left\{\min\left\{\frac1{K^j\mu^2_j} , \frac{{(\lambda^2_j)}^p} {K^j\mu^2_j}\right\}\right\}\leq \sup_j\left\{\frac1{K^j\mu^2_j}\right\}=1<+\infty.$$

Next, we will construct a function $f\in N^{1, p}(X,ds^2,d\mu^2)$ such that $\Tr f$ does not exist. 
For $j\in\N$, let $2\beta_j=1/\lambda^2_j=2^{-j}\leq 1$. For any edge $[x, y]$ on $X$ with $|x|=j$ and $|y|=j+1$, we define the value of $f$ on the point $z_1\in [x, y]$ with $|z_1|=j+\beta_k$ as $1$. Let $z_2\in[x, y]$ be the point with $|z_2|=j+2\beta_k$. Then we define $f$ to be linear on $[x, z_1]$ and $[z_1, z_2]$ with $f(t)=0$ for any $t\in \{x\}\cup[z_2, y]$. Then we know that $|f|\leq 1$, $g_f\leq 1/(\beta_j\lambda^2_j)=2$ on $[x, z_2]$ and that $|f|=g_f=0$ on $[z_2, y]$. Hence we obtain the estimate
$$\int_{X_2} |f|^p+g_f^p\, d\mu^2\lesssim \sum_{j\in\N} {2\beta_j\mu_j^2 K^j} =2 \sum_{j\in\N} \frac{\mu_j^2 K^j}{\lambda_j^2}  =2\sum_{j\in\N} 2^{-j}<+\infty.$$
Thus, $f\in N^{1, p}(X_2, ds^2, d\mu^2)$. Moreover, for any $\xi\in\bx$, by the definition of $f$, it is easy to find two sequences $\{x_n\}_{n\in\N}$ and $\{y_n\}_{n\in\N}$ on the geodesic $[0, \xi)$ such that
$$\lim_{n\rightarrow \infty} f(x_{n})=0\ \ \text{and \ \ } \lim_{n\rightarrow \infty} f(y_{n})=1,$$
which implies that $\lim_{[0, \xi)\ni x\rarrow \xi} f(x)$ {does not exist}, {for any} $\xi\in \bx$. Hence $\Tr f$ does not exist.
\end{example}

The above results of Newtonian space $N^{1, p}(X)$ with  $\mu(X)=\infty$ and $p>1$ lead to an open question: For $\mu(X)=\infty$ and $p>1$, what conditions (only depending on $\lambda$, $\mu$ and $p$) give a full characterization of the existence of traces for $N^{1, p}(X)$?

\noindent{\bf }

\section{Density}\label{density-results}
In this section, we focus on the density properties of compactly supported functions in $N^{1, p}(X)$ and in $\dot N^{1, p}(X)$, $1\leq p<\infty$. The function $\mathbf{1}$ is defined by $\mathbf{1}(x)=1$ for all $x$ in $X$ and we abuse the notation by using $\nabla u$ to denote $g_u$ if needed for convenience.

Our first result is an analog of the corresponding result for 
infinite networks \cite{Y77}, also see \cite{ns}.

\begin{lem}\label{lemma1} Let $1 \leq p<\infty$ and assume that $\mu(X)<\infty$. Then we have that
\[
N^{1,p}_0(X)=N^{1,p}(X) \iff
\mathbf{1}\in N^{1,p}_0(X).
\]
\end{lem}
\begin{proof}
Since it follows from  $\mu(X)<\infty$ that $\mathbf{1}\in N^{1,p}(X)$, we obtain that $N^{1,p}_0(X)=N^{1,p}(X)$  implies $\mathbf{1}\in N^{1,p}_0(X)$.

Towards the other direction, the hypothesis $\mathbf{1}\in N^{1,p}_0(X)$ gives a  family of compactly supported functions $\{\mathbf{1}_n\}_{n\in\mathbb{N}}$   in $N^{1,p}(X)$ such that $\mathbf{1}_n\to \mathbf{1}$ { in }$ N^{1,p}(X)$ as $n\rightarrow \infty$. Recall that $X^m:=\{x\in X: |x|<m\}$ for any $m\in \N$. Without loss of generality we may assume that $\mathbf{1}_n$ is nonnegative for any $n\in \N$ and that 
$$\|\mathbf{1}_n-\mathbf{1}\|^p_{N^{1, p}(X)}< \frac{1}{4^p}\mu(X^1),\ \ \forall \ n\in\N.$$
We claim that for any $n\in \N$, there exists a point $x_n$ with $x_n\in X^1$ such that $|1-\mathbf{1}_n(x_n)|<1/4$. If not, then we have $|1-\mathbf{1}_n(x)|\geq\frac{1}{4}$ for any $x\in X^1$. Hence we obtain that
\[\| 1-u\|^p_{N^{1,p}(X)}\geq \| 1-u\|^p_{N^{1,p}(X^1)} \geq \frac{1}{4^p} \mu(X^1),\]
which is a contradiction. By the triangle inequality, we have  $1-\mathbf{1}_n(x_n)\leq |1-\mathbf{1}_n(x_n)|<\frac 14$, and hence $\mathbf{1}_n(x_n)>\frac34$. 

Next, we claim that we may assume $\mathbf{1}_n(x)>1/2$ for all $x\in X^n$ by selecting a subsequence of $\{\mathbf{1}_n\}_{n\in\mathbb{N}}$ if necessary. Assume 
that this claim is not true. Then there exists $N\in \N$ such that for any 
$n\in \N$, there exists a point $y_n\in X^N$ with $\mathbf{1}_n(y_n)\leq 1/2$. Hence for any $n\in \N$, we have found two points $x_n, y_n\in X^N$ such that $|\mathbf{1}(x_n)_n-\mathbf{1}_n(y_n)|\geq 1/4$. Let $\gamma=[x_n, y_n]$ be the geodesic connecting $x_n$ and $y_n$. Then 
$$\int_\gamma \nabla(\mathbf{1}_n)\, ds\geq 1/4\ \ \text{for any}\ \ n\in \N.$$
By an argument similar to that for the estimate \eqref{Mod_p} and \eqref{Mod_1}, we have that there exists a constant $C(N, p, \lambda,\mu)>0$ such that
 \[\int_X|\nabla(\mathbf{1}-\mathbf{1}_n)|^p\, d\mu= \int_X |\nabla(\mathbf{1}_n)|^p\, d\mu\geq C(N, p, \lambda, \mu)>0\ \ \text{for any}\ \ n\in \N,\]

which is a contradiction to $\mathbf{1}_n\rightarrow \mathbf{1}$ in $N^{1, p}(X)$.

Thus, from the arguments above, we may assume that there exists a family of compactly supported functions $\{\mathbf{1}_n\}_{n\in\mathbb{N}}$   in $N^{1,p}(X)$ such that
$$\begin{cases}\mathbf{1}_n\to \mathbf{1} \text{ in } N^{1,p}(X)\text{ as }n\to \infty,\\
\mathbf{1}_n(x)\geq \frac{1}{2} \text{ for any } x\in X^n .
\end{cases}$$

We define $\mathbf{\bar 1}_n:=\min\{2 \cdot \mathbf{1}_n,\mathbf{1}\}$ for all $n\in \mathbb{N}$. Then the family $(\mathbf{\bar 1}_n)_{n\in\mathbb{N}}$ satisfies 
\begin{equation}\label{1nto1}
\begin{cases}
\mathbf{\bar 1}_n\to \mathbf{1} \text{ in }N^{1,p}(X) \text{ as }n\to\infty,\\
\mathbf{\bar 1}_n\equiv 1 \text{ in }X^n \\
\mathbf{\bar 1}_n \text{ is a function with compact support.  } 
\end{cases}
\end{equation}
Given a function $u$ in $N^{1,p}(X)$, let us show that $u_n \mathbf{\bar 1}_n \rightarrow u$ in $N^{1,p}(X)$ where $u_n(x)$ is a truncation  of $u$ with respect to $a_n:=\|\mathbf{\bar 1}_n-\mathbf{1}\|^{-1/2}_{N^{1,p}(X)}$, namely
$$u_n(x)=
\begin{cases}\frac{u}{|u|}a_n & \text{ if }|u|\geq a_n\\
u&\text{ if }|u|\leq a_n
\end{cases}.
$$
From the basic properties of truncation (see for instance \cite[Section 7.1]{pekka}), we have that 

\begin{equation}
\label{untou}
\begin{cases}u_n\to u \text{ in }N^{1,p}{(X)} \text{ as }n\to\infty,\\
|u_n(x)|\leq a_n , \\
|\nabla u_n|\leq 3|\nabla u|.
\end{cases}
\end{equation}

We first show that $u_n\mathbf{\bar 1}_n\to u$ in $L^p(X)$ as $n\to \infty$. By the triangle inequality, it follows from \eqref{1nto1} and \eqref{untou} that
\begin{align*}
\|u_n\mathbf{\bar 1}_n-u\|_{L^{p}(X)}& \leq \|u_n\mathbf{\bar 1}_n-u_n \|_{L^{p}(X)}+\|u_n-u\|_{L^{p}(X)}\\
&\leq  a_n \|\mathbf{\bar 1}_n-\mathbf{1}\|_{N^{1,p}(X)}+\|u_n-u\|_{L^{p}(X)}\\
&= \| \mathbf{\bar{1}_n-\mathbf{1}}\|_{N^{1,p}(X)}^{1/2}+\|u_n-u\|_{L^{p}(X)}\rightarrow 0\  \text{ as } \ n\rightarrow\infty.
\end{align*}
Recall that every function in $N^{1, p}(X)$ is locally absolutely continuous, see Section \ref{newtonian}. By the product rule of locally absolutely continuous functions, we obtain that 
  \begin{align*}
 |\nabla (u_n\mathbf{\bar{1}_n}-u)|&=|\nabla (u_n\mathbf{\bar 1}_n-u_n+u_n-u)|\\
 &\leq |u_n||\nabla (\mathbf{\bar 1}_n-\mathbf{1})|+|\mathbf{\bar 1}_n-\mathbf{1}||\nabla u_n|+|\nabla (u_n-u)|\\
 &\leq a_n |\nabla (\mathbf{\bar 1}_n-\mathbf{1})|+|\nabla u_n|\chi_{X\setminus X^n}+|\nabla (u_n-u)|.
 \end{align*}
Hence we obtain from the triangle inequality that
\begin{align*}
\|\nabla (u_n\mathbf{\bar 1}_n-u)\|_{L^{p}(X)}& \leq  
 a_n\|\nabla (\mathbf{\bar 1}_n-\mathbf{1})\|_{L^{p}(X)}+\|\nabla u_n\|_{L^{p}(X\setminus X^n)}+\|\nabla (u_n-u)\|_{L^{p}(X)}\\
&\leq \|\mathbf{\bar 1}_n-\mathbf{1}\|^{1/2}_{N^{1,p}(X)}+3\|\nabla u\|_{L^{p}(X\setminus X^n)}+\|\nabla( u_n-u)\|_{L^{p}(X)},
\end{align*}
which tends to $0$ as $n\to \infty$. Therefore, $u_n\mathbf{\bar{1}}_n\to u$ in $N^{1,p}(X)$ as $n\to \infty$. Since the support of $u_n\mathbf{\bar{1}}_n$ is compact, it follows from the definition of $N^{1,p}_0(X)$ that $u\in N^{1,p}_0(X)$, and hence $N^{1,p}_0(X)=N^{1, p}(X)$.
\end{proof}

Notice that  $\mathbf{1}\in \dot N^{1, p}(X)$ no matter if $\mu(X)$ is finite or not. By slightly modifying the previous proof, we obtain the following result.
\begin{cor}\label{cor1}
Let $1 \leq p<\infty$. Then the following statements are equivalent
\[
\dot N^{1,p}_0(X)=\dot N^{1,p}(X)\iff
\mathbf{1}\in \dot N^{1,p}_0(X)
\]
\end{cor}

Applying Lemma \ref{lemma1}, we obtain our first density result.

\begin{prop}\label{lemma2} Let $1 \leq p<\infty$ and assume that  $\mu(X)<\infty$.
	Suppose additionally that $R_p=\infty$. Then we have that
	$$N^{1, p}_{0}(X)={N}^{1, p}(X).$$
\end{prop}
\begin{proof}

    It follows from Lemma \ref{lemma1} that 
	it suffices to construct a sequence of compactly supported $N^{1, p}$-functions  which converges to $\mathbf{1}$ in $N^{1, p}(X)$.

	For $p>1$ and 
	\begin{equation}\notag 
	R_p=\int_{0}^\infty \lambda^{\frac{p}{p-1}}(t)\mu^{\frac{1}{1-p}}(t)K^{\frac{t}{1-p}}dt=\infty,
	\end{equation}
	
	we define the family of functions $\{\varphi_n\}_{n\in\mathbb{N}}$ as follows. For each $n \in \mathbb{N}$, let $r_n>n$ be an integer such that 
	\begin{equation}\label{value-r_n}
	\int_{n}^{r_n}\lambda^{\frac{p}{p-1}}(t)\mu^{\frac{1}{1-p}}(t)K^{\frac{t}{1-p}}dt\geq 2^n.
	\end{equation}
We set $\varphi_n (x)=1$ for all $x\in X^n$, $\varphi_n(x)=0$ for all $x\in X\setminus X^{r_n}$ and 
	$$
	\varphi_n(x)=1-\frac{\int_{n}^{|x|}\lambda^{\frac{p}{p-1}}(t)\mu^{\frac{1}{1-p}}(t)K^{\frac{t}{1-p}}dt}{\int_{n}^{r_n}\lambda^{\frac{p}{p-1}}(t)\mu^{\frac{1}{1-p}}(t)K^{\frac{t}{1-p}}dt}
	$$ for all $x\in  X^{r_n}\setminus X^n$. Since $\lambda^p/\mu\in L^{1/(p-1)}_{\rm loc}([0, \infty))$ and $\lambda, 1/\mu>0$, then $\varphi_n$ is well-defined.
     It is easy to check that $\varphi_n$ is compactly supported.
	
	By the construction of $\varphi_n$, an easy computation shows that 
	\begin{equation}\label{nablaun1}
	\nabla (\varphi_n (x)-1)=0 
	\end{equation}
	for all $x\in (X^n)\cup (X\setminus X^{r_n})$ and that 
	\begin{equation}\label{nablaun2}
	|\nabla (\varphi_n(x)-1)|\leq \frac{1}{\lambda(x)}\frac{\lambda^{\frac{p}{p-1}}(x)\mu^{\frac{1}{1-p}}(x)K^{\frac{|x|}{1-p}}}{\int_{n}^{r_n}\lambda^{\frac{p}{p-1}}(t)\mu^{\frac{1}{1-p}}(t)K^{\frac{t}{1-p}}dt}=\frac{\lambda^{\frac{1}{p-1}}(x)\mu^{\frac{1}{1-p}}(x)K^{\frac{|x|}{1-p}}}{\int_{n}^{r_n}\lambda^{\frac{p}{p-1}}(t)\mu^{\frac{1}{1-p}}(t)K^{\frac{t}{1-p}}dt}
	\end{equation}
	for all $x\in X^{r_n}\setminus X^{n}$. 
	
	Thanks to \eqref{nablaun1} and \eqref{nablaun2}, we obtain the estimate
	\begin{align*}
	\int_X|\nabla(\varphi_n-\mathbf{1})|^pd\mu&= \int_{X^{r_n}\setminus X^n}|\nabla(\varphi_n-\mathbf{1})|^pd\mu\\
	&\approx \int_{n}^{r_n}K^{t}\mu(t)\left(\frac{\lambda^{\frac{1}{p-1}}(t)\mu^{\frac{1}{1-p}}(t)K^{\frac{t}{1-p}}}{\int_{n}^{r_n}\lambda^{\frac{p}{p-1}}(t)\mu^{\frac{1}{1-p}}(t)K^{\frac{t}{1-p}}dt}\right)^pdt\\
	&=\left(\int_{n}^{r_n}\lambda^{\frac{p}{p-1}}(t)\mu^{\frac{1}{1-p}}(t)K^{\frac{t}{1-p}}dt\right)^{1-p}.
	\end{align*}
	Since $p>1$ and \eqref{value-r_n} holds, we obtain that	
	$$\left(\int_{n}^{r_n}\lambda^{\frac{p}{p-1}}(t)\mu^{\frac{1}{1-p}}(t)K^{\frac{1}{1-p}}dt\right)^{1-p}\to 0\ \text{as}\  n\to \infty.$$
	Hence we have that  $\|\nabla (\varphi_n-\mathbf{1})\|_{L^{p}(X)}\to 0$ as $n\to\infty$. Moreover, since $|\varphi_n-\mathbf{1}|\leq 2\chi_{X\setminus X^n}$, it follows from  $\mu(X)<\infty$ that 
	$$
	\|\varphi_n(x)-\mathbf{1}\|_{L^p(X)}\leq 2 \mu(X\setminus X^n)\to 0\ \text{as}\  n\to \infty.
	$$
	Therefore, $\varphi_n \to \mathbf{1}$ in $N^{1,p}(X)$ as $n\to\infty$. 
	
	For $p=1$, since $\lambda/\mu\in L^\infty_{\rm loc}([0, \infty))$ implies that 
 $\left\|\frac{\lambda(t)}{\mu(t)K^{ t }}\right\|_{L^{\infty}([0, n) )}<\infty$ for any $n\in \N$, it follows from 
$R_1=\left \| \frac{\lambda(t)}{\mu(t)K^{t}}\right \|_{{L^1({[0,\infty)})}}=\infty$ that  the sequence of sets 
	\[E_k:=\left \{t\in [k,\infty): \frac{\lambda(t)}{\mu(t)K^{}}\geq 2^k \right \}, \ \ k\in\mathbb{N} 
	\] is a nonincreasing sequence of subset of $[0, \infty)$ and that we have 
\[|E_k|>0\ \ \text{for any}\ \ k\in\mathbb N.\]
	
	We have $E_k= \lim_{n\to \infty} E_k\cap[k,n]$ and $|E_k|=\lim_{n\to \infty}|E_k\cap[k,n]|$. Hence there exist a $k_n>k$ such that
		\[E_{k_n}:=\left \{t\in [k,k_n]: \frac{\lambda(t)}{\mu(t)K^{t}}\geq 2^k \right \} 
	\] satisfies $0<|E_{k_n}|<\infty$. 
	
	We define a sequence $\{\varphi_{k}\}$ of functions  by setting
	\[\varphi_{k}(x)=1-\frac{1}{|E_{k_n}|}\int_{k}^{|x|} \chi_{E_{k_n}}(t)dt
	\] for all $|x|\in[k,k_n]$ and $\varphi_{k}(x)=0$ on $X\setminus X^{k_n}$, $\varphi_{k}(x)=1$ on $X^k$. 
	
	It follows directly from the definition of $\varphi_{k}$ that each $\varphi_{k}$ has compact support and that
	\[|\varphi_{k}-1|\leq 2 \chi_{X\setminus X^k}, \ \ |\nabla (\varphi_{k})(x)-1)|\leq \frac{\chi_{E_{k_n}}(x)}{\lambda(x)|E_{k_n}|}.
	\]
	Hence, thanks to $\mu(X)<\infty$ and the definition of $E_{k_n}$, we obtain  that
	\begin{align*}
	\| \varphi_{k}-1\|_{N^{1,1}(X)}&= \| \varphi_{k}-1\|_{L^1(X)}+ \| \nabla(\varphi_{k}-1)\|_{L^{1}(X)}\\
	&\leq \| 2\chi_{X\setminus X^{k}}\|_{L^{1}(X)}+ \int_X \frac{\chi_{E_{k_n}}(t)}{\lambda(t)|E_{k_n}|}d\mu(t)\\
	&\lesssim 2\mu(X\setminus X^k)+\frac{1}{|E_{k_n|}}\int_{k}^{k_n} \frac{\mu(t)K^{t}}{\lambda(t)} \chi_{E_{k_n}}(t)\,dt\\
	&\leq 2\mu(X\setminus X^k)+ \frac{1}{2^k} \to 0 \text{ as }k\to \infty.
	\end{align*}
    Hence $\varphi_{k}\rightarrow \mathbf 1$ in $N^{1, 1}(X)$ as $k\rightarrow \infty$.
\end{proof}

By using the same construction of the sequence of compactly supported $N^{1, p}$-functions as the one in the above proof, we obtain the following corollary immediately from Corollary \ref{cor1}.

\begin{cor}\label{dot-N-dense}
Let $1 \leq p<\infty$. Assume $R_p=\infty$. Then we obtain that
\[\dot N^{1, p}_0(X)= \dot N^{1, p}(X).\]
\end{cor}

\begin{prop}\label{N-subset}
Let  $1 \leq p<\infty$ and assume that  $\mu(X)<\infty$. Suppose additionally that $R_p<\infty$. Then we have 
$${N}_{0}^{1, p}(X)\subsetneq{N}^{1, p}(X).$$
\end{prop}

\begin{proof}
Suppose  ${N}_{0}^{1, p}(X)=N^{1,p}(X).$ Since $1\in N^{1,p}(X)$, it follows that for every $\varepsilon>0$, there exists a function $u\in N^{1, p}(X)$ with compact support such that 
\begin{equation}\label{1-u}
\| 1-u\|_{N^{1,p}(X)}<\varepsilon.
\end{equation}
Let $\xi\in\partial X$ be arbitrary, and $x_j:=x_j(\xi)$ be the ancestor of $\xi$ with $|x_j|=j$ and $x_0=0$.  Let $0<\epsilon<\frac 12\|\mu\|^{1/p}_{L^1([0, 1])}$. By repeating the argument in the beginning of Proof of Lemma \ref{lemma1} 
with the change that we replace $\mu(X^1)/4^p$ and $X^1$ by $\epsilon^p$ and 
$[0, x_1(\xi)]$, respectively,   we obtain the existence of  $x_\xi\in [0, x_1(\xi)]$ for which the function $u$ in \eqref{1-u} satisfies $|1-u(x_\xi)|<\frac{1}{2}$.  By the triangle inequality, we have $1-|u(x_\xi)|\leq |1-u(x_\xi)|<\frac 12$, and hence $|u(x_\xi)|>\frac12$.

Notice that $u$ has compact support. Then for any $\xi\in\bx$, we have $\lim_{n\rightarrow \infty} u(x_n(\xi))=0$ and that 
\begin{equation}\label{exam-estimate}
\frac12< \lim_{n\rightarrow \infty}|u(x_\xi)-u(x_n(\xi))| \leq \int_{[0,\xi)} g_u\, ds=\int_{[0, \xi)} g_u \frac{\lambda(x)}{\mu(x)}\, d\mu.
\end{equation}

For $p=1$, integrating over all $\xi\in\bx$, since $\nu(\bx)=1$, we obtain by Fubini's theorem that
\begin{align*}
\frac 12 &\leq \int_{\bx} \int_{X} g_u(x) \chi_{[0,\xi)}(x)\frac{\lambda(x)}{\mu(x)} \, d\mu(x)\, d\nu(\xi)\\
& =\int_X g_u(x) \frac{\lambda(x)}{\mu(x)} \left(\int_\bx \chi_{[0,\xi)}(x) \, d\nu(\xi)\right)\, d\mu(x)\\
& = \int_X g_u(x) \lambda(x)\mu(x)^{-1} \nu(I_x)\, d\mu(x),
\end{align*}
where $I_x=\{\xi\in \bx: x<\xi\}$. Since $\nu(I_X) \approx K^{-|x|}$, we obtain from $R_1<\infty$ that 
\[\frac 12\lesssim R_1 \int_X g_u(x)\, d\mu(x) \lesssim \|1-u\|_{N^{1,1}(X)}<\epsilon,\]
where the second to last inequality holds since every upper gradient of $u$ is also an upper gradient of $1-u$. By choosing $\epsilon$ small enough, the above estimate yields a contradiction, and hence $N^{1,1}_0(X) \not= N^{1, 1}(X)$.

For $p>1$, by \eqref{exam-estimate} and the H\"older inequality, we have  that 
\begin{align*}
\frac{1}{2^p} &\leq \left(\int_{[0, \xi)} {g_u} K^{j(x)/p} \frac{\lambda(x)}{\mu(x)K^{j(x)/p}}  \, d\mu\right)^p\\
& \leq \int_{[0, \xi)} {g_u}^p K^{j(x)}\, d\mu \left(\int_{[0, \xi)} \left(\frac{\lambda(x)}{\mu(x)K^{j(x)/p}}\right)^{\frac{p}{p-1}}\, d\mu\right)^{p-1}\\
&\le  {R_p}^{p-1} \int_{[0, \xi)} {g_u}^p K^{j(x)}\, d\mu,
\end{align*}
where $j(x)$ is the largest integer such that $j(x)\leq |x|$. Here the last inequality holds since
\[\int_{[0, \xi)} \left(\frac{\lambda(x)}{\mu(x)K^{j(x)/p}}\right)^{\frac{p}{p-1}}\, d\mu\approx \int_{0}^\infty \frac{\lambda(t)^{\frac{p}{p-1}}}{\mu(t)^{\frac{p}{p-1}}K^{\frac{t}{p-1}}} \mu(t)\, dt=R_p.\]
Integrating over all $\xi\in \bx$, since $\nu(\bx)= 1$ and $R_p<+\infty$, we obtain by Fubini's theorem that
\begin{align*}
\frac{1}{2^p}&\lesssim \int_\bx \int_{X} {g_u(x)}^p \chi_{[0,\xi)}(x) K^{j(x)}\, d\mu(x)\, d\nu(\xi) \\
&= \int_X g_u(x)^p K^{j(x)}\left(\int_\bx \chi_{[0, \xi)}(x)\, d\nu(\xi)\right)\, d\mu(x)\\
&= \int_X g_u(x)^p K^{j(x)} \nu(I_x)\, d\mu(x),
\end{align*}
where the notations $I_x$ and $j(x)$ are the same ones as those we used before.
Since $\nu(I_x)\approx K^{-j(x)}$, we obtain the estimate
\[\frac{1}{2} \lesssim \left(\int_{X} g_u(x)^p\, d\mu(x) \right)^{1/p}\leq \| 1-u\|_{N^{1,p}(X)} <\epsilon,\]
where the second to last inequality holds since every upper gradient of $u$ is also an upper gradient of $1-u$. By choosing $\epsilon$ small enough, the above estimate gives a contradiction, and hence $N^{1,p}_0(X) \not= N^{1, p}(X)$ for $p>1$.

Since $N^{1,p}_0(X) \subset N^{1,p}(X)$ by definition, we obtain $N^{1, p}_{0}(X) \subsetneq N^{1, p}(X)$ for all  $p\geq 1$.

\end{proof}

\begin{cor}\label{dot-N-subset}
Let $p\geq 1$ and assume that $R_p<\infty$. Then we have
\[\dot N^{1,p}_{0}(X)\subsetneq \dot N^{1,p}(X).\]
\end{cor}
\begin{proof}
Suppose $\dot N^{1,p}_{0}(X)= \dot N^{1,p}(X)$. Since $1\in \dot N^{1,p}(X)$, it follows that for every $\epsilon>0$, there exists a function $u\in \dot N^{1, p}(X)$ with compact support such that
\[\|1-u\|_{\dot N^{1, p}(X)}<\epsilon.\]
Then by the definition of our $\dot N^{1, p}$-norm, we have
$|u(0)-1|<\epsilon$ and hence $|u(0)|>1-\epsilon$. 

Then using an argument similar to the one in the proof of Proposition \ref{N-subset} (replace $u(x_\xi)$ with $u(0)$), we obtain a contradiction. The claim follows.
\end{proof}

The above results give a full picture for the density properties for homogeneous Newtonian spaces $\dot N^{1, p}(X)$ and for Newtonian spaces $N^{1, p}(X)$ when $\mu(X)<\infty$. When the total measure is infinite, the density results for the
Newtonian space $N^{1, p}(X)$ are quite different.
\begin{lem}\label{K_1_dense} 
Let $K=1$, i.e., $X$ be a $1$-regular tree and assume that $\mu(X)=\infty$. Then for any $f\in N^{1, p}(X)$, there exists a sequence of compactly supported $N^{1, p}$-functions $\{f_n\}_{n\in\N}$ such that $f_n\rightarrow f$ in $N^{1, p}(X)$.
\end{lem}
\begin{proof}
Notice that we may compose any $f\in N^{1,p}(X)$ as $f=f^{+}-f^{-}$ where $f^{+}=f\chi_{\{|f|\geq 0\}}\geq 0$ and $f^{-}=-f\chi_{\{|f|\leq 0\}}\geq 0$. Hence we may assume that $f\geq 0$.

Since $K=1$, $\bx$ contains only one point $\xi_0$ and there is a unique geodesic ray. It follows from Proposition \ref{trace-0} that 
\begin{equation}\label{lower-limit}
\liminf_{[0, \xi_0)\ni x\rightarrow \xi_0} f(x)=0.
\end{equation}
Denote by $x_n$ the vertex of $X$ with $|x_n|=n$ when $n\in \N$. Then it follows from \eqref{lower-limit} that 
\begin{equation}\label{one-lower}
f(x_n)-\int_{[x_n,\xi_0)} g_f\, ds\leq 0,\ \ \forall \ n\in \N.
\end{equation}
We define functions $f_n$ by setting
$$f_n(x):=
\begin{cases} f(x), &\text{ if }|x|\leq n;\\
\max\{0, f(x_n)-2\int_{[x_n,x]}g_f\,ds\}, & \text{ if } |x|> n.
\end{cases}
$$
Then it is easy to check that $f_n\in N^{1, p}(X)$, since $0\leq f_n\leq f$ and $g_{f_n}\leq 2 g_f$. Next, we check that $f_n$ is compactly supported. Assume not. Since $f_n$ is non-increasing for $|x|>n$ by definition, we have that $f_n(x)>0$ for any $|x|>n$ and hence that 
\[\lim_{x\rightarrow \xi_0} f_n(x)=f(x_n)-2\int_{[x_n, \xi_0)} g_f\, ds\geq 0.\]
Combining this with \eqref{one-lower}, we conclude that 
\[\int_{[x_n, \xi_0)} g_f\, ds =0.\]
Then $g_f=0$ for $|x|>n$ and it follows from \eqref{lower-limit} that $f$ has to be identically $0$ for $|x|\geq n$, which is a contradiction. Hence $f_n$ is compactly supported.

At last, we estimate the $N^{1, p}$-norm of $f_n-f$. By the fact that $0\leq f_n\leq f$ and $g_{f_n}\leq 2 g_f$, we obtain the estimate 
\begin{align*}
\|f_n-f\|_{N^{1, p}(X)}&=\|f_n-f\|_{N^{1, p}({X\cap\{|x|\geq n\}})}\\
&\leq \|f_n\|_{N^{1, p}({X\cap\{|x|\geq n\}})}+\|f\|_{N^{1, p}({X\cap\{|x|\geq n\}})}\\
&\leq 3\|f\|_{{N^{1, p}({X\cap\{|x|\geq n\}})}}\rightarrow 0 \ \ \text{as}\ \ n\rightarrow 0,
\end{align*}
since $f\in N^{1, p}(X)$. Thus $\{f_n\}_{n\in \N}$ is a sequence of compactly supported $N^{1, p}$-functions with $f_n\rightarrow f$ in $N^{1, p}(X)$, which finishes the proof.
\end{proof}
If $\int_0^{\infty}\lambda(t)\, dt=\infty$, then $X$ is complete and unbounded with respect to distance $d$ and it follows by using the suitable cutoff functions that $N^{1, p}_0(X)=N^{1, p}(X)$. Our next result shows that this is also the case when $X$ is bounded and not complete if we assume $\mu(X)=\infty$.

\begin{thm}\label{dense-infinity}
Let  $1 \leq p<\infty$ and assume that $\mu(X)=\infty$. Then we have that
\[N^{1, p}_{0}(X)=N^{1,p}(X).\]
\end{thm}
\begin{proof}
If $K=1$, the result follows directly from Lemma \ref{K_1_dense}. Hence we assume $K\geq 2$ in the ensuing proof.

For any $f\in N^{1, p}(X)$, by the same argument as in Lemma \ref{K_1_dense}, we may assume that $f\geq 0$. It suffices to construct a sequence $\{f_n\}_{n\in \N}$ of compactly supported $N^{1, p}$-functions such that $f_n\rightarrow f$ in $N^{1, p}(X)$. 

For each $n\in \N$, we denote by $\{x_{n, j}\}_{j=1}^{K^n}$ the vertices of $n$-level, i.e., $|x_{n, j}|=n$ for all $j=1, \cdots, K^n$. For any $x_{n, j}$, we study the subtree $\Gamma_{x_{n, j}}$ which is a subset of $X$ with root $x_{n, j}$. More precisely,   
\[\Gamma_{x_{n, j}}:=\{x\in X: x_{n, j}<x\}.\]
Since every vertex has exactly $K$ children, we may divide $\Gamma_{x_{n, j}}$ into $K$ subsets, where each subset contains a subtree whose root is a child of $x_{n, j}$ and  an edge connecting this child with $x_{n, j}$. We denote by $\{\Gamma^i_{x_{n, j}}\}_{i=1}^{K}$ these $K$ subsets.

Fix $f\in N^{1, p}(X)$. We first study the function $u:=f|_{\Gamma_{x_{n, j}}}$. If $\|f\|_{N^{1, p}(\Gamma_{x_{n, j}})}>0$,  we first modify the function $u$ to a function $v$ with $v(x)=v(|x|)$ for any $x\in \Gamma_{x_{n, j}}$, i.e., for any $x, y\in \Gamma_{x_{n, j}}$ with $|x|=|y|$, then $v(x)=v(y)$. The modification procedure is as follows:
\item[\underline{Step 1}]
Since $\Gamma_{x_{n, j}}=\bigcup_{i=1}^{K} \Gamma_{x_{n, j}}^i$, without loss of generality, we may assume
\begin{equation}
\label{ugam1}
\left\|u\right\|_{N^{1,p}(\Gamma_{x_{n, j}}^1)}=\min\{\|u\|_{N^{1,p}(\Gamma_{x_{n, j}}^i)}:i=1,2,\ldots,K\}.
\end{equation}
Then we define a function $u^1$ by identically copying the minimal $N^{1, p}$-energy subtree of $u$ (here is $u|_{\Gamma^1_{x_{n, j}}}$), to the other $k-1$ subtrees $\Gamma_{x_{n, j}}^i$,  $i=2, \cdots K$. More precisely, 
\begin{equation}
\notag
u^1(x):=
\begin{cases}
u(x),& \text{ if }x\in \Gamma_{x_{n, j}}^1;\\
u|_{\Gamma_{x_{n,j}}^1}(y) \ \text{with}\ y\in \Gamma_{x_{n, j}}^1, |y|=|x|,& \text{ if }x\in \Gamma_{x_{n, j}}^i.
\end{cases}
\end{equation}  
 It follows from \eqref{ugam1} that
\begin{equation}\notag
\|u^1\|_{N^{1, p}(\Gamma_{x_{n, j}})}\leq \|u\|_{N^{1, p}(\Gamma_{x_{n, j}})}.
\end{equation}
Then for any $x, y\in \Gamma_{x_{n, j}}\cap\{x\in X: n\leq|x|\leq n+1\}$  with $|x|=|y|$, we have $u^1(x)=u^1(y)$.

\item[\underline{Step 2}]
Denote by $\{x_{n+1,t}\}_{t=1}^K$ the $K$ children of $x_{n, j}$. We repeat the Step 1 by replacing the function $u$ and $\Gamma_{x_{n, j}}$ with $u^1$ and $\Gamma_{x_{n+1,t}}$, respectively. Here we repeat the Step 1 for all $K$ subtrees $\Gamma_{x_{n+1,t}}$, $t=1, \cdots, K$. Hence we obtain a function $u^2$ on $\Gamma_{x_{n, j}}$ by additionally letting $u^2(x)=u^1(x)$ if $x\in \Gamma_{x_{n, j}}$ with $n\leq|x|\leq n+1$.
Moreover, it is easy to check that 
\[\|u^2\|_{N^{1, p}(\Gamma_{x_{n, j}})}\leq \|u^1\|_{N^{1, p}(\Gamma_{x_{n, j}})}\leq \|u\|_{N^{1, p}(\Gamma_{x_{n, j}})} \]
and that $u^2(x)=u^2(y)$ for any $x, y\in \Gamma_{x_{n, j}}\cap\{x\in X: n\leq|x|\leq n+2\}$  with $|x|=|y|$.

Continuing this procedure, we obtain a sequence of functions $\{u^k\}_{k\in \N}$. We define $v=\lim_{k\rightarrow \infty} u^k$. Then we know from induction that 
\begin{equation}\label{estimate-v}
\|v\|_{N^{1, p}(\Gamma_{x_{n, j}})}\leq \|u\|_{N^{1, p}(\Gamma_{x_{n, j}})}=\|f\|_{N^{1, p}(\Gamma_{x_{n, j}})}
\end{equation}
and that $v(x)=v(y)$ for any $x, y\in \Gamma_{x_{n, j}}$ with $|x|=|y|$.

The value of function $v(x)$ only depends on the distance $d(x_{n, j}, x)$. We may regard $v$ as a function on a $1$-regular tree with root $x_{n, j}$ and infinite measure, since $\mu(\Gamma_{x_{n,j}})=\infty$. Hence, from the proof of Lemma \ref{K_1_dense}, we are able to choose a compactly supported $N^{1, p}$-function $f_{{n, j}}$ on $\Gamma_{x_{n, j}}$ with 
\begin{equation}\label{estimate-upper}
\|f_{n, j}-v\|_{N^{1, p}(\Gamma_{x_{n, j}})}\leq \|u\|_{N^{1, p}(\Gamma_{x_{n, j}})}=\|f\|_{N^{1, p}(\Gamma_{x_{n, j}})}.
\end{equation}
Then it follows from \eqref{estimate-v} and \eqref{estimate-upper} that
\begin{align}
\|f_{n, j}-f\|_{N^{1, p}(\Gamma_{x_{n, j}})}&\leq \|f_{n, j}-v\|_{N^{1, p}(\Gamma_{x_{n, j}})}+\|v-f\|_{N^{1, p}(\Gamma_{x_{n, j}})}\notag\\
&\leq \|f_{n, j}-v\|_{N^{1, p}(\Gamma_{x_{n, j}})}+\|v\|_{N^{1, p}(\Gamma_{x_{n, j}})}+\|f\|_{N^{1, p}(\Gamma_{x_{n, j}})}\notag\\
&\leq 3\|f\|_{N^{1, p}(\Gamma_{x_{n, j}})}.\label{estimate-v-f}
\end{align}

If $\|f\|_{N^{1, p}(\Gamma_{x_{n, j}})}=0$, then $f=0$ on $\Gamma_{x_{n, j}}$ and we just define $f_{x_{n, j}}=f|_{\Gamma_{x_{n, j}}}$.

At last, we define a function $f_n$ by setting
$$f_n(x):=
\begin{cases} f(x), &\text{ if }|x|\leq n;\\
f_{{n, j}}(x), & \text{ if } x\in \Gamma_{x_{n, j}}.
\end{cases}
$$ 
Then it is easy to check that $f_n\in N^{1, p}(X)$ and that $f_n$ is compactly supported, since $f_{{n, j}}$ are compactly supported for any $j=1, \cdots, K^n$. 
It follows from estimate \eqref{estimate-v-f} that
\begin{align*}
\|f_n-f\|_{N^{1, p}(X)}&=\|f_n-f\|_{N^{1, p}(X\cap \{|x|\geq n\})}=\sum_{j=1}^{K^n} \|f_n-f\|_{N^{1, p}(\Gamma_{x_{n, j}})}\\
&=\sum_{j=1}^{K^n} \|f_{n, j}-f\|_{N^{1, p}(\Gamma_{x_{n, j}})}\leq 3\sum_{j=1}^{K^n}\|f\|_{N^{1, p}(\Gamma_{x_{n, j}})}\\
&=3 \|f\|_{N^{1, p}(X\cap \{|x|\geq n\})} \rightarrow 0,\ \ \text{as} \ n\rightarrow 0,
\end{align*}
since $f\in N^{1, p}(X)$. Thus we have found  a sequence $\{f_n\}_{n\in \N}$
of compactly supported $N^{1, p}$-functions with $f_n\rightarrow f$ in $N^{1, p}(X)$, which finishes the proof.
\end{proof}

\section{Proofs of Theorems}\label{proofs}
\begin{proof}[Proof of Theorem \ref{main-thm1}]
$(i) \Rightarrow (ii)$ is given by Theorem \ref{trace}; $(ii)\Rightarrow (iii)$ is trivial and $(iii)\Rightarrow (i)$ is given by Theorem \ref{trace-example}.

$(i)\Rightarrow (iv)$ is given by Corollary \ref{dot-N-subset} and $(iv)\Rightarrow (i)$ is given by Corollary \ref{dot-N-dense}.
\end{proof}

\begin{proof}[Proof of Theorem \ref{main-thm2}]
$(i) \Rightarrow (ii)$ is given by Corollary \ref{trace-pp}; $(ii)\Rightarrow (iii)$ is trivial and $(iii)\Rightarrow (i)$ is given by Proposition \ref{trace-N-example}.

$(i)\Rightarrow (iv)$ is given by Proposition \ref{N-subset} and $(iv)\Rightarrow (i)$ is given by Corollary \ref{lemma2}.
\end{proof}

\begin{proof}[Proof of Theorem \ref{main-infinity}]
The statement (1) and statement (2) follow from Theorem \ref{dense-infinity}
and Proposition \ref{trace-0} respectively. The statement (3) follows by  combining Corollary \ref{trace-pp} and Proposition \ref{no-trace-infinity}. Corollary \ref{trace-pp} gives the statement (4). 
\end{proof}

\section*{Acknowledgments}

The authors would like to thank Prof. Nageswari Shanmugalingam for reading the manuscript and giving comments that helped improve the paper.

\noindent
Department of Mathematics and Statistics, University of Jyv\"askyl\"a, PO~Box~35, FI-40014 Jyv\"askyl\"a, Finland.

\medskip

\noindent Pekka Koskela

\noindent{\it E-mail address}:  \texttt{pekka.j.koskela@jyu.fi}

\noindent Khanh Ngoc Nguyen

\noindent{\it E-mail address}:  \texttt{khanh.n.nguyen@jyu.fi}, \texttt{khanh.mimhus@gmail.com}

\noindent Zhuang Wang

\noindent{\it E-mail address}:  \texttt{zhuang.z.wang@jyu.fi}

\end{document}